\documentclass[twoside,12pt]{article}
\usepackage{amsmath}
\usepackage{amssymb}
\usepackage{amsthm}
\usepackage{mathrsfs}
\usepackage{latexsym}

\DeclareMathAlphabet{\mathpzc}{OT1}{pzc}{m}{it}

\newtheorem{theorem}{Theorem}[section]
\newtheorem{lemma}[theorem]{Lemma}
\newtheorem{corollary}[theorem]{Corollary}
\newtheorem{proposition}[theorem]{Proposition}
\newcommand{\abs}[1]{\lvert#1\rvert}
\newcommand{\li}[1]{{\rm li} (#1)}
\newcommand{\Li}[1]{{\rm Li} (#1)}

\title{{\scshape Group-theoretic remarks on Goldbach's conjecture}}
\author{Liguo He \thanks{Corresponding author. E-mail address:  helg-lxy@sut.edu.cn} \quad Xianyu Hu
\\{\footnotesize Dept. of Math., Shenyang University of Technology, Shenyang, 110870,
 PR China}}
\date{}

\begin{document}
\maketitle

\noindent \textbf{Abstract.} {\footnotesize{The famous strongly binary Goldbach's conjecture asserts that every even number $2n \geq 8$ can always be expressible as the sum of two distinct odd prime numbers. We use a new approach to dealing with this conjecture. Specifically, we apply the element order prime graphs of alternating groups of degrees $2n$ and $2n-1$ to characterize this conjecture, and present its six group-theoretic versions; and further prove that this conjecture is true for $p+1$ and $p-1$ whenever $p \geq 11$ is a prime number.
}



\bigskip

\section{Introduction}

The famous \textit{Strongly Binary Goldbach Conjecture}\cite{gold} asserts that for every even number $2n \geq  8$, there exist distinct odd primes $p, q$ with $2n = p + q$. If this situation does occur,  $2n$ is called \textit{a Goldbach's number}. The conjecture is a well-known unsolved problem dating from 1742 due to C. Goldbach. It is commonly considered as an extremely difficult problem of analytic number theory these days. Considering the fundamental role of finite groups (especially, the alternating group $A_n$ of degree $n \geq 5$) in solving the radical solution problem of polynomial equations of degree $5$ or more (due to E. Galois, for example, see \cite{Is2}), we are inspired to attack the Goldbach's problem by appealing to the finite group theory and especially, the alternating group $A_n$ of degree $n \geq 8$. The following  \textsc{Theorem B} partially confirms our guess (although its proof is short), it also shows that there exist infinitely many Goldbach's numbers. It is achived  via \cite{gold} that all even numbers $2n \leq 4 \times 10^{18}$ are Goldbach's numbers (as of the year 2013) except possibly when $n$ is a prime.

Let G be a finite group and $\pi(G)$ the set of prime factors of its order. The element order prime graph $\Gamma (G)$ of G is a graph whose vertex-set $\mathcal{V}(G)$ is just $ \pi(G)$, and two vertices $p, q$ are joined by an edge whenever G contains an element of order pq. The edge set of $\Gamma(G)$ is denoted by $\mathcal{E}(G)$. This graph is also referred to as  Gruenberg-Kergel graph of $G$. Regarding this graph, we refer to \cite{L53,W22} for more detailed information. For groups $G_1$ and $G_2$, if $\mathcal{V}(G_1) \subseteq \mathcal{V}(G_2)$ and $\mathcal{E}(G_1) \subseteq \mathcal{E}(G_2)$, then $\Gamma(G_1)$ is said to be a subgraph of $\Gamma(G_2)$ and denoted by $\Gamma(G_1) \leq \Gamma(G_2)$. Furthermore, if $\mathcal{V}(G_1)$ is a proper subset of $\mathcal{V}(G_2)$, or $\mathcal{E}(G_1)$ is a proper subset of $\mathcal{E}(G_2)$, then $\Gamma(G_1)$ is called a proper subgraph of $\Gamma(G_2)$ and written as $\Gamma(G_1) < \Gamma(G_2)$. Following convention, we write $\pi(x)$ for the prime-counting function which stands for the number of primes not exceeding the positive real number $x$. By $A_n$, denote the alternating group of degree $n$. For a Sylow $p$-subgroup $P$ of $A_n$, $C_{A_n}(P)$ (resp. $N_{A_n}(P)$) indicates the centralizer (resp. normalizer) of $P$ in $A_{n}$. For two sets  $S_1$ and $S_2$, the difference set $S_2- S_1 = \{x \mid \, x \in S_2$ but $x \not\in S_1 \}$, whose cardinality is indicated by $\abs{S_2- S_1}$.

In this paper, we prove the following results.

\textsc{Theorem A \ }{\it Assume the above notation and the even integer $2n \geq 8$. Then the following assertions are equivalent.}
{\it
\begin{enumerate}
\item The even number $2n$ is a Goldbach's number.

\item $\Gamma(A_{2n-1})$ is a proper subgraph of $\Gamma(A_{2n})$.

\item $\abs{\mathcal{E}(A_{2n}) - \mathcal{E}(A_{2n-1})} \geq 1$.

\item $\abs{\mathcal{E}(A_{2n}) - \mathcal{E}(A_{2n-1})} \geq 1$ when both $\Gamma(A_{2n-1})$ and $\Gamma(A_{2n})$ are connected graphs.

\item  $\abs{\pi(C_{A_{2n}}(P)) - \pi(C_{A_{2n-1}}(P))} \geq 1$ for some odd prime $p$ with $n < p \leq 2n-3$ and some $P \in Syl_p(A_{2n-1})$.
\item  $\abs{\pi(N_{A_{2n}}(P)) - \pi(N_{A_{2n-1}}(P))} \geq 1$ for some odd prime $p$ with $n < p \leq 2n-3$ and some $P \in Syl_p(A_{2n-1})$.
\item $ dim\mathscr{U}\hspace{-0.5mm}(A_{2n}) > dim \mathscr{U}\hspace{-0.5mm}(A_{2n-1})$ for the biprimary spaces $\mathscr{U}\hspace{-0.5mm}(A_{2n})$ and $\mathscr{U}\hspace{-0.5mm}(A_{2n-1})$.
\end{enumerate}
}
Actually the edge number difference $\abs{\mathcal{E}(A_{2n}) - \mathcal{E}(A_{2n-1})}$ is exactly the number of expressions of $2n$ as sum of two distinct odd primes. This expression number seems to be limitless as $n$ approaches infinity. The inequality $\abs{\mathcal{E}(A_{2n}) - \mathcal{E}(A_{2n-1})} \geq 1$ also implies that $A_{2n}$ and $A_{2n-1}$ can be recognizable each other by prime graphs, but in general it is impossible between $A_{2n+1}$ and $A_{2n}$. For instance, if $4 \leq n \leq 29$, it can be  verified by GAP \cite{gap} that $\abs{\mathcal{E}(A_{2n+1}) - \mathcal{E}(A_{2n})} = 0$ just when $n \in \{ 6, 9, 12, 14, 15,18, 19, 21, 24, 26, 27, 29 \}$. (For the GAP command codes, see the appendix.) We mention that Theorem 1 in \cite{Gold-2} shows  that $A_n$ can be characterized by the full set of its element orders when $n \geq 5, n \neq 6, 10$.

Applying the prime graphic approach of finite groups, we may prove a class of even numbers to be Goldbach's numbers.

\textsc{Theorem B \ }{\it Assume that any $p \geq 11$ is an odd prime. Then the Strongly Binary Goldbach Conjecture is true for $p+1$ and $p-1$.}

Unless otherwise stated, the notation and terminology is standard, as presented in \cite{Is, Kurz, L46}.

\section{Prime graph}

The following observation is a basic but crucial fact, which appears as Proposition 1.1 in \cite{V5} without proof there. We restate it in the language of prime graph.
\begin{lemma}\label{v5}
Let $A_{n}$ denote the alternating group of degree $n \geq 8$.
\begin{enumerate}
\item For distinct odd primes $p, q \in \mathcal{V}(A_{n})$, the edge $pq \in \mathcal{E}(A_{n})$ if and only if $p + q \leq n$.
\item For distinct primes $2, p \in \mathcal{V}(A_{n})$, the edge $2p \in \mathcal{E}(A_{n})$ if and only if $p + 4 \leq n$.
\end{enumerate}
\end{lemma}
\begin{proof}
Let $p, q$ be different odd primes. If $p + q \leq n$, it is no loss to pick the element $(1, 2, \cdots, p)(p+1, \cdots, p+q) \in A_n$. If $p + 4 \leq n$, it is no loss to choose the element $(1, 2, \cdots, p)(p+1, p+2)(p+3, p+4) \in A_{n}$, the ``if" parts are obtained.

If the edge $pq \in \mathcal{E}(A_{n})$, then $A_{n}$ contains an element $x$ of order $pq$, and $x$ has disjoint cycle product expression $x = c_1 c_2 \cdots c_t$ with cycle lengths $\abs{c_i} = m_i$ and the order of $x$ is the least common multiple $[m_1, m_2, \cdots, m_t]$ which equals $pq$, and thus $m_i = p^{\alpha}q^{\beta}$ with $0 \leq \alpha, \beta \leq 1$.  Since also $\Sigma_{i = 1}^{t} m_i \leq n$, we get $p + q \leq n$. (This can also be attained by Corollary 1 of \cite{M52}.) If $A_{n}$ has an element $x$ of order $2p$, then its disjoint cycle product expression contains at least two even cycles (i.e., their cycle lengths are even numbers), the even cycles are either $2$-cycles or $2p$-cycles, hence $p + 4 \leq n$,  the ``only if" parts are achieved.
\end{proof}

The following is Part 2 of Theorem A.

\begin{theorem}\label{goldb3}
The Strongly Binary Goldbach's conjecture is true for $2n (\geq 8)$ if and only if the element order prime graph $\Gamma(A_{2n-1})$ is a proper subgraph of $\Gamma(A_{2n})$.
\end{theorem}
\begin{proof}
For odd prime $p$, $p + 4 \leq 2n$ if and only if $p + 4 \leq 2n - 1$. Thus Lemma \ref{v5} yields that $2p$ is an edge of  $\Gamma(A_{2n-1})$ if and only if it is  also an edge of $\Gamma(A_{2n})$. Hence if $\Gamma(A_{2n-1})$ is a proper subgraph of $\Gamma(A_{2n})$, then there exist distinct odd primes $p, q$ such that the edge $ pq \in \mathcal{E}(A_{2n})$ but $ pq \not\in \mathcal{E}(A_{2n-1})$, thus Lemma \ref{v5} implies $p + q \leq 2n$ and $p+q > 2n-1$, hence we have $2n = p + q$, as desired. The reverse statement is immediate by noting the fact $\pi(A_{2n}) = \pi(A_{2n-1})$.
\end{proof}

The following is Part 3 of Theorem A.

\begin{corollary}\label{goldb4}
The Strong Binary Goldbach's conjecture is valid for $2n (\geq 8)$ if and only if \\ $\abs{\mathcal{E}(A_{2n}) - \mathcal{E}(A_{2n-1})} \geq 1$.
\end{corollary}
\begin{proof}
 Since $\pi(A_{2n}) = \pi(A_{2n-1})$,  we reach that $\Gamma(A_{2n-1})$ is a proper subgraph of $\Gamma(A_{2n})$ if and only if $\mathcal{E}(A_{2n-1})$ is a proper subset of $\mathcal{E}(A_{2n})$, that is, $\abs{\mathcal{E}(A_{2n}) - \mathcal{E}(A_{2n-1})} \geq 1$. Thus the desired result follows from Theorem \ref{goldb3}.
\end{proof}
Because $\abs{\mathcal{E}(A_{2n}) - \mathcal{E}(A_{2n-1})} \geq 0$ is an integer,  it is easy to see that  $\abs{\mathcal{E}(A_{2n}) - \mathcal{E}(A_{2n-1})} \geq 1$ if and only if $\abs{\mathcal{E}(A_{2n}) - \mathcal{E}(A_{2n-1})} > 0$. A little bit of difference between the two expressions is of meaningful sometimes in order to prove this conjecture when the method of analytic number theory is applied.

Let the alternating group $A_n$ permute the symbol set $ \Omega = \{1, 2, \cdots, n\}$ and write $A_t (\leq A_n)$ as $A_{(s, s+ t)}$ in order to indicate the moved elements to be at most in the symbol subset $\{s, s+1, \cdots, s+t\} \subseteq \Omega$. For $ s+t < r < r+ m \leq n$, write $ A_{(s, s+t)} \times A_{(r, r+m)} < A_n$ to denote the inner direct product of $ A_{(s, s+t)}$ and $A_{(r, r+m)}$ in $A_n$.
\begin{lemma}\label{cean}
Let $q$ be a prime with $n/2 < q \leq n-2$ and $n \geq 8$, and let $Q = \langle  x \rangle \in Syl_q(A_n)$ for some element $x \in A_n$ of order $q$. Assume that $Q$ just permutes the symbol subset $\{1, 2, \cdots, q\}$. Then  $\abs{C_{A_n}(x)} = \abs{C_{A_n}(Q)} = q\frac{(n-q)!}{2}$ and  $\abs{N_{A_n}(Q)} = (q-1)q\frac{(n-q)!}{2}$. Furthermore, $C_{A_n}(x) = C_{A_n}(Q) = Q \times A_{(q+1, n)}$ and  $N_{A_n}(Q) = (Q \rtimes C_{q-1})\times A_{(q+1, n)}$ where $C_{q-1}$ is a cyclic subgroup of $A_{(1, q)}$ with order $q-1$ and the subgroup $Q\rtimes C_{q-1}$ is a Frobenius group.
\end{lemma}
\begin{proof}
For the element $x \in Q$ of order $q$, then it is a $q$-cycle and so $C_{A_n}(x) = C_{A_n}(Q)$. Using the crucial observation  $x^g = (i_1^g, i_2^g, \cdots, i_q^g)$ for any $g \in A_n$, we may deduce $C_{A_n}(Q) = Q \times A_{(q+1, n)}$ and so $\abs{C_{A_n}(Q)} = q\frac{(n-q)!}{2}$. Note that if $x = (1, 2, \cdots, q)$, then $x^g = x$ implies $$(1^g, 2^g, \cdots, q^g) = (1, 2, 3, \cdots, q) = (q, 1, 2, \cdots, q-1) = \cdots = (2, 3, \cdots, q, 1).$$

For $g \in N_{A_n}(Q)$, we have $x^g = x^k \in Q$ for some $1 \leq k \leq q-1$, and $x^k$ is still a $q$-cycle. If there exists $h \in N_{A_n}(Q)$ such that $x^h = x^k$, then $x^{hg^{-1}} = x$ and so $hg^{-1} \in C_{A_n}(Q)$ and $h \in gC_{A_n}(Q)$. Indeed these further implies that $N_{A_n}(Q) \leq A_{(1, q)} \times A_{(q+1, n)}$. Note that we have $gC_{A_n}(Q) = C_{A_n}(Q)g$ for $g \in N_{A_n}(Q)$ as $C_{A_n}(Q) \trianglelefteq  N_{A_n}(Q)$

For $1 \leq m \leq q-1$, each of $x$ and $x^m$ is $q$-cycle, so both of them are conjugate, i.e., $x^m = x^h$, $h \in S_n$ (which denotes the symmetric group of degree n). We may choose the element $h$ in $A_n$. As $x \in A_{(1,q)}$, we get $x^m \in A_{(1,q)}$. If $h$ is an odd permutation, we may replace $h$ with the disjoint cycle product $yh$, where $y = (q+1,q+2)$ and $q \leq n-2$. We see $x^{yh} = x^h = x^m$. For $h \in A_n$, we may further derive that $h \in A_{(1, q)}\times A_{(q+1, n)}$ and $h \in N_{A_n}(Q)$, this is because $x^m$ and $x$ lie in $Q$.

 For $1 \leq m < n \leq q-1$, let $x^m = x^{g_1}$ and $x^n = x^{g_2}$ for $g_1, g_2 \in A_n$, we claim $C_{A_n}(Q)g_1 \neq C_{A_n}(Q)g_2$. If otherwise, $x^{g_1} = x^{g_2}$ and so $x^n = x^m$, then $x^{n-m} = 1$, this is a contradiction since the order of $x$ is $q$. Hence we deduce $\abs{N_{A_n}(Q)/C_{A_n}(Q)} = q-1$. The N/C Theorem further yields $N_{A_n}(Q)/C_{A_n}(Q) \cong C_{q-1}$.
Since $N_{A_n}(Q) \leq A_{(1, q)}\times A_{(q+1, n)}$, it follows via Dedekind identity that $$N_{A_n}(Q) = (N_{A_n}(Q) \cap A_{(1, q)})\times A_{(q+1, n)} = N_{A_{(1, q)}}(Q)\times A_{(q+1, n)}.$$
Since also $Q$ is a normal Sylow $q$-subgroup of $N_{A_{(1, q)}}(Q)$ and $\abs{Q} = q$, we can derive that $N_{A_{(1, q)}}(Q) = Q \rtimes C_{q-1}$ for some $C_{q-1} < A_{(1, q)}$. Note that the following equality $$ \frac{\abs{N_{A_{(1, q)}}(Q)}}{\abs{Q}} = \frac{\abs{N_{A_{(1, q)}}(Q)}\abs{A_{(q+1, n)}}}{\abs{Q}\abs{A_{(q+1, n)}}} = \abs{N_{A_n}(Q)/C_{A_n}(Q)} =  q-1.$$
Therefore, we conclude that $N_{A_n}(Q) = (Q \rtimes C_{q-1}) \times A_{(q+1, n)}.$  Because $C_{A_{(1, q)}}(Q) = Q$, we reach that $C_{q-1}$ acts fixed-point-freely on $Q$, it follows via \cite[Theorem 8.1.12]{Kurz} that $Q \rtimes C_{q-1}$ is a Frobenius group. The proof is complete.
\end{proof}

Observe that the above result actually implies both $N_{A_{2n}}(Q)$ and $N_{A_{2n-1}}(Q)$ have the same Frobenius subgroups of order $pq$.

\begin{lemma}\label{edgenumber}
Let the natural number $n \geq 4$ and the primes $q$ with $n < q \leq 2n-3$, then
\begin{enumerate}
\item The number of edges incident to vertices $q$ of $\Gamma(A_{2n})$ is equal to $\underset{n < q \leq 2n-3}{\sum} |\pi(A_{2n - q})|$ which equals $\underset{n < q \leq 2n-3}{\sum}\pi({2n - q})$.

\item The number of edges incident to vertices $q$ of $\Gamma(A_{2n-1})$ is equal to $\underset{n < q \leq 2n-3}{\sum} |\pi(A_{2n -1 - q})|$ which equals $\underset{n < q \leq 2n-3}{\sum}\pi({2n -1 - q})$.
\end{enumerate}
\end{lemma}
\begin{proof}
For $n < q \leq 2n-3$, if $\mathcal{E}(A_{2n})$ contains edge $pq$, then $A_{2n}$ has an element $g$ of order $pq$ which can be uniquely written in the form $g = g_pg_p$ with $p$-part $g_p$ and $q$-part $g_q$ of $g$, thus $g_p \in C_{A_{2n}}(g_q)$, Lemma \ref{cean} yields that $p \in \pi(A_{2n-q})$. Conversely, if $p \in \pi(A_{2n-q})$, then $p+ q \leq 2n$ for odd $p$ and $4 + q \leq 2n$ for $p = 2$, Lemma \ref{v5} yields $A_{2n}$ has an element $g$ of order $pq$. Note that if $2 \in \pi(A_{2n-q})$, then $2$ divides $(2n-q)!/2$, thus $2n-q \geq 5$ so that $4 + q < 2n$. It is straightforward that $$\underset{n < q \leq 2n-3}{\sum} |\pi(A_{2n - q})| = \underset{n < q \leq 2n-3}{\sum}\pi({2n - q}).$$  Part 1 follows.  And Part 2 can be derived in a similar manner.
\end{proof}

\begin{theorem}\label{edgeequ}
Let the natural number $n \geq 4$, then
$$\abs{\mathcal{E}(A_{2n}) - \mathcal{E}(A_{2n-1})} = \underset{3 \leq p < n}{\sum} (\pi(2n - p) - \pi(2n-1-p)).$$
\end{theorem}
\begin{proof}
By Lemma \ref{edgenumber}, we see $$\abs{\mathcal{E}(A_{2n}) - \mathcal{E}(A_{2n-1})} = \underset{n < q \leq 2n-3}{\sum}(\pi({2n - q})- \pi({2n -1 - q})).$$
It is easy to see that $\pi({2n - q})- \pi({2n -1 - q})$ equals either $1$ or else $0$. If $\pi({2n - q})- \pi({2n -1 - q})= 1$, then there exists a unique odd $p < n$ such that the edge $pq \in \mathcal{E}(A_{2n}) - \mathcal{E}(A_{2n-1})$, Lemma \ref{v5} yields $ 2n-1 < p + q \leq 2n$ and so $q = 2n - p$, thus $q \in \pi(A_{2n-p})$ but $q \not\in \pi(A_{2n-1-p})$, hence $\pi(2n-p) - \pi(2n-1-p) = 1$, and vice versa. Therefore we conclude that

$$ \underset{n < q \leq 2n-3}{\sum}(\pi({2n - q})- \pi({2n -1 - q})) = \underset{3 \leq p < n}{\sum} (\pi(2n - p) - \pi(2n-1-p)),$$ yielding the desired result.
\end{proof}

The prime number theorem (PNT for short) with the best known error term is $$ \pi(x) = \li{x} + O(x \exp(-c(\log x)^{3/5}(\log\log x)^{-1/5}))$$ for some strictly positive constant $c$. This can be found in\cite[p250]{Cohen2} and the definition of $\li{x}$ is in \cite[p257]{Cohen2} where $\li{x}$ is denoted $\Li{x}$. Under the Riemman's hypothesis, it is known via \cite{gold} that the PNT has a concise form $\abs{\pi(x)-\li{x}} < \frac{\sqrt{x}\log x}{8\pi}$. We may set $\pi(x) = \li{x} + h(x)\frac{\sqrt{x}\log x}{8\pi}$ for some function $h(x)$ satisfying $\abs{h(x)} < 1$. Thus it seems natural to estimate the above expression of Theorem \ref{edgeequ} by using the PNT, but it is necessary to deeply understand $h(x)$.  Note that in order to prove $\abs{\mathcal{E}(A_{2n}) - \mathcal{E}(A_{2n-1})} \geq 1$, it is enough to prove $\abs{\mathcal{E}(A_{2n}) - \mathcal{E}(A_{2n-1})} > 0$. For any enough small positive number $\varepsilon$, we may let $\bar{h}(x) = \varepsilon h(x)$, then $\abs{\bar{h}(x)} < \varepsilon$ and the PNT has form $\pi(x) = \li{x} + \bar{h}(x)\frac{\sqrt{x}\log x}{8\pi\varepsilon}$, which seems more useful. Without Riemann's hypothesis, we may also present a similar form for $\pi(x)$.
Although these observations  seem to be interesting, we are still intent to handle the conjecture by appealing to the method associated closely with the finite group theory in this paper.

\begin{lemma}\label{iso}
Let $n \geq 6$.  If $\Gamma(A_{2n}) = \Gamma(A_{2n-1})$, then $\Gamma(A_{2n}) = \Gamma(A_{2n-2})$.
\end{lemma}
\begin{proof}
It is evident that $\Gamma(A_{2n-2}) \leq \Gamma(A_{2n})$. Under the hypothesis above, we shall prove  $\Gamma(A_{2n}) \leq \Gamma(A_{2n-2})$.

 Assume that $A_{2n}$ contains an element of odd order $pq$. Since $\Gamma(A_{2n}) = \Gamma(A_{2n-1})$, it follows that $A_{2n - 1}$ also contains an element of odd order $pq$, then  $p + q \leq 2n - 1$ by Lemma \ref{v5}, consequently $p + q \leq 2n - 2$ (as the sum $p + q$ is even), thus we conclude that $A_{2n - 2}$ has an element of odd order $pq$.

Assume now that $A_{2n}$ has an element $x$ of order $2p$. Then the product expression of disjoint cycles of $x$ contains at most some $2$-cycles, $p$-cycles or $2p$-cycles. If its expression has all cycles of three types, then $2 + p + 2p \leq 2n$ and so $4 + p \leq 2n - 2$ (as $p \geq 3$), Lemma \ref{v5} implies that $A_{2n-2}$ owns elements of order $2p$. Because a single even cycle is an odd permutation, it follows that if the expression exactly contains one type of cycles, then the only possibility is of $2p$-cycles, we get $2(2p) \leq 2n$ and so $4 + p \leq 2n-2$, as wanted. Therefore we are reduced to the case where the expression precisely contains two types of cycles.

If $x$ is a product of some $2$-cycles and $p$-cycles, then the number of $2$-cycle factors in the product expression of disjoint cycles of $x$ is even number, say $2t$. When $t \geq 2$, we know $4 + p \leq 2n - 2$ (as $2t\cdot 2 + p \leq 2n$). When $t = 1$, we see that the odd number $4 + p \leq 2n - 1$. If $A_{2n - 2}$ contains no element of order $2p$, then $4 + p > 2n - 2$ (again by Lemma \ref{v5}), hence $4 + p = 2n - 1$, that is, $2n = 5 + p$. Since $n \geq 6$, we get that $p \neq 5$, and thus $\Gamma(A_{2n})$ has edge $5p$ but $\Gamma(A_{2n-1})$ has not, hence $\Gamma(A_{2n-1}) < \Gamma(A_{2n})$, a contradiction. Hence $A_{2n-2}$ contains elements of order $2p$, as desired.
If $x$ is a product of some $2$-cycles and $2p$-cycles, then we obtain that $2 + 2p \leq 2n$ and so $ 4 + p \leq 2n - 1$, thus the same argument as the preceding paragraph yields the desired result.

If $x$ is a product of some $p$-cycles and $2p$-cycles, then $p + 2(2p) \leq 2n$ and so $4 + p \leq 2n - 2$, as required. The proof is finished.
\end{proof}

The following is the former part of \textsc{Theorem B}.

\begin{theorem}\label{go-sp1}
Let $p \geq 7 $ be a prime, then $p+1$ is a Goldbach's number.
\end{theorem}
\begin{proof}
It is evident that $ 8 = 5 + 3$, thus we may assume $p \geq 11$.
If $\Gamma(A_{p+1}) = \Gamma(A_p)$, then Lemma \ref{iso} yields $\Gamma(A_{p+1}) = \Gamma(A_{p-1})$. However, this is impossible since $\Gamma(A_{p+1})$ has vertex $p$, which is not a vertex of $\Gamma(A_{p-1})$. Thus we obtain that $\Gamma(A_{p+1}) > \Gamma(A_p)$, then Lemma \ref{v5} yields that $p+1$ is just a Goldbach's number, as desired.
\end{proof}

\begin{theorem}\label{Goldb5}
Let $2n \geq 8$.  If $\Gamma(A_{2n}) < \Gamma(A_{2n+1})$, then $2n = p + 3$ for some odd prime $p > 3$.
\end{theorem}
\begin{proof}
  For the distinct odd primes $p, q$, if the edge $pq \in \mathcal{E}(A_{2n+1})$, then $p + q \leq 2n+1$, and so $p+q \leq 2n$, Lemma \ref{v5} yields the edge $pq \in \mathcal{E}(A_{2n})$. Thus there exists some edge $2p \in \mathcal{E}(A_{2n+1})$ but not in $\mathcal{E}(A_{2n})$. Lemma \ref{v5} shows that $4 + p \leq 2n + 1$ but $4+p > 2n$, which forces $2n+1 = p+4$, and so $2n = p + 3$. Also $2n \geq 8$ and so $p > 3$, it follows that $2n$ is a Goldbach's number, as desired.
\end{proof}

The following is the latter part of \textsc{Theorem B}.

\begin{theorem}\label{go-sp2}
Let $p \geq 11 $ be a prime, then $p-1$ is a Goldbach's number.
\end{theorem}
\begin{proof}
Since  $p \in \mathcal{V}(A_p)$ but not in $\mathcal{V}(A_{p-1})$,  it follows that $\Gamma(A_{p-1}) < \Gamma(A_{p})$, then Theorem \ref{Goldb5} yields that $p-1$ is a Goldbach's number, as wanted.
\end{proof}

\begin{proposition}\label{pr-di}
It is true that $\pi(x) - \pi(6x/7) \geq 1$ for $x \geq 37$.
\end{proposition}
\begin{proof}
See Theorem 2 of \cite{Ross}.
\end{proof}

The next consequence shows that there are infinitely many Goldbach's numbers.

\begin{corollary}
For each $n \geq 5$, there exists at least two Goldbach's numbers $2m-2, 2m$ satisfying $\frac{12n-7}{7} < 2m-2, 2m \leq 2n$.
\end{corollary}
\begin{proof}
By Proposition \ref{pr-di}, there is a prime $p$ with $\frac{12n}{7} < p < 2n$ for $n \geq 19$. Applying Theorems \ref{go-sp1} and \ref{go-sp2},  we get that $p-1, p+1$ are Goldbach's numbers and $ \frac{12n-7}{7} < p-1, p+1 \leq 2n$, we may take $2m - 2 = p-1$ and $2m = p+1$. For $5 \leq n \leq 18$, it is routine to check that there exist Goldbach's numbers $2m-2, 2m$ satisfying  $ \frac{12n-7}{7} < 2m-2, 2m \leq 2n$, as claimed.
\end{proof}

The following result reduces  the Strongly Binary Goldbach's conjecture to the situation where both graphs  $\Gamma(A_{2n})$ and $\Gamma(A_{2n-1})$ connected, which is Part 4 of \textsc{Theorem A}.
\begin{theorem}\label{reduc}
Let $n \geq 4$, the graph $\Gamma(A_{2n-1})$ or $\Gamma(A_{2n})$ is disconnected, then $2n$ is a Goldbach's number.
\end{theorem}
\begin{proof}
By Theorem 1 of \cite{W22}, the element order prime graphs of alternating groups on five or more symbols have at most three components. Table Id of \cite{W22} implies that $\Gamma(A_{2n})$ can not have three components. If  $\Gamma(A_{2n})$ has two components, then Table Ib of \cite{W22} implies $2n = p + 1$ for odd prime $p$, the result follows from Theorem \ref{go-sp1}.
If $\Gamma(A_{2n})$ has one component and $\Gamma(A_{2n-1})$ has two components, then the application of Theorem \ref{goldb3} yields the result.
\end{proof}

\begin{theorem}\label{smanum}
It is valid that $1 \leq \abs{\mathcal{E}(A_{2n}) - \mathcal{E}(A_{2n-1})}\leq 6$  when $4 \leq n \leq 30$.
\end{theorem}
\begin{proof}
By using GAP \cite{gap}, we may compute that the edge number differences $\abs{\mathcal{E}(A_{2n}) - \mathcal{E}(A_{2n-1})}$ when $4 \leq n \leq 30$. (For GAP command codes, see the appendix.)
Set $d(n) = \abs{\mathcal{E}(A_{2n}) - \mathcal{E}(A_{2n-1})}$, the specific results are listed in the following Table \ref{tb1}.
\begin{table}[ht]
\caption{Edge number differences $d(n)$ for $4 \leq n \leq 30$}\label{tb1}
\begin{center}
\begin{tabular}{rrrrrrrrrrrrrrr}
\hline
n&4&5&6&7&8&9&10&11&12&13&14&15&16&17\\
d(n)&1&1&1&1&2&2&2&2&3&2&2&3&2&3\\
\hline
n&18&19&20&21&22&23&24&25&26&27&28&29&30&\\
d(n)&4&1&3&4&3&3&5&4&3&5&3&3&6&\\
\hline
\end{tabular}
\end{center}
\end{table}
\end{proof}

\section{Centralizer}
We use $C_G(g)$ to denote the centralizer of $g$ in $G$, i.e. $C_G(g) = \{ x \in G \mid \, xg = gx \}$.

\begin{theorem}\label{cenver}
The even number $2n \geq 8$ is a Goldbach's number if and only if there exists an element  $g \in A_{2n-1}$ of odd prime order such that  $\pi(C_{A_{2n-1}}(g))$ is a proper subset of $\pi(C_{A_{2n}}(g))$.
\end{theorem}
\begin{proof}
Set $A_{2n}$ to act on the symbol set $\Omega = \{1, 2, \cdots, 2n-1, 2n\}$; and $A_{2n-1}$ on the symbol set $\Omega_1 = \{1, 2, \cdots, 2n-1\}$.
Suppose that $2n$ is a Goldbach's number. Then $2n = s + t$ for distinct odd primes $s > t$. Pick $g = (1, 2, \cdots, s) \in A_{2n-1}$ and so $x = (s+1, s+2, \cdots, 2n) \not\in A_{2n-1}$, we have $x \in C_{A_{2n}}(g)$. If $t \in \pi(C_{A_{2n-1}}(g))$, then since $(t, s) = 1$, it follows via Lemma \ref{cean} that there is an element $(a_1, a_2, \cdots, a_t) \in C_{A_{2n-1}}(g)$ satisfying all $ s+1 \leq a_i \leq 2n-1$, which forces $t + s \leq 2n-1$, this contradiction shows $t \not\in \pi(C_{A_{2n-1}}(g))$.   Conversely, if $\pi(C_{A_{2n-1}}(g))$ is a proper subset of $\pi(C_{A_{2n}}(g))$ for some element $g$ (in $A_{2n-1}$) of order $s$, then there exists prime $t \in \pi(C_{A_{2n}}(g))- \pi(C_{A_{2n-1}}(g))$, thus $C_{A_{2n}}(g)$ contains element $x$ of order $t$, but $C_{A_{2n-1}}(g)$ contains no element of order $t$. Note that $t$ must be an odd prime. Hence $A_{2n}$ has an element, say $gx$, of order $st$. Note that $(s, t) = 1$. However, $A_{2n-1}$ does not contain any element of order $st$. If this is not the case, let $z \in A_{2n-1}$ be of order $st$, then $z^t$ is of order $s$ and conjugate to $g$, say $z^t = g^h$ for some $h \in A_{2n}$, and $z^s$ has order $t$. Thus $ C_{A_{2n-1}^h}(g^h)$ contains the element $z^s$ of order $t$. This is a contradiction since $$\pi(C_{A_{2n-1}}(g)) = \pi((C_{A_{2n-1}}(g))^h) = \pi(C_{A_{2n-1}^h}(g^h)).$$ Hence $A_{2n}$ contains element of order $st$, but not for $A_{2n-1}$, the application of Lemma \ref{v5} yields $2n = s + t$, as required.
\end{proof}

For any prime $n < p < 2n$, the Sylow $p$-subgroup $P$ of $A_{2n-1}$ is of order $p$, which is also a Sylow $p$-subgroup of $A_{2n}$, denoted by $P \in Syl_p(A_{2n})$. Hence Theorem \ref{cenver} can be expressed as the following version, which is Part 5 of \textsc{Theorem A}.
\begin{corollary}\label{cenver2}
The even number $2n \geq 8$ is a Goldbach's number if and only if there exists an odd prime $ n < p \leq 2n-3$ such that $\abs{\pi(C_{A_{2n}}(P)) - \pi(C_{A_{2n-1}}(P))} \geq 1$ for some $P \in Syl_p(A_{2n-1})$.
\end{corollary}
\begin{proof}
Following from Theorem \ref{cenver}. Note that $C_{A_{2n-1}}(g) = C_{A_{2n-1}}(P)$ and  $C_{A_{2n}}(g) = C_{A_{2n}}(P)$ for $1 \neq g \in P$, $P \in Syl_p(A_{2n-1})$ and $n < p \leq 2n-3$.
\end{proof}

\begin{corollary}\label{cenver3}
The even number $2n \geq 8$ is a Goldbach's number if and only if there exists an odd prime $n < q \leq 2n-3$ such that $C_{A_{2n}}(Q)$ contains elements of order $pq$ but $C_{A_{2n-1}}(Q)$ does not contains elements of order $pq$ for some $q \neq p \in \pi(A_{2n-1})$.
\end{corollary}
\begin{proof}
Immediate from Corollary \ref{cenver3}.
\end{proof}

For the distinct odd primes $p, q$, it is easy to see that the group $G$ has elements of order $pq$ if and only if $G$ has a cycle subgroup of order $pq$. However, even if $G$ has a subgroup of order $pq$, $G$ need not necessarily contain elements of order $pq$. For the primes $p > q$ with $p \equiv 1 \mod q$, we may construct the semidirect product $G = P \rtimes Q$ with $\abs{P} = p $, $ \abs{Q} = q$ and $Q \leq Aut(P)$. Here $G$ is indeed a Frobenius group of order $pq$. The following is a general result, which is a direct consequence of G. Higman's theorem in \cite{Higman}.

\begin{theorem}\label{pqt}
Let $G$ be a $pq$-group. Then $G$ has no element of order $pq$ if and only if $G$ is a Frobenius group.
\end{theorem}
\begin{proof}
The $pq$-group $G$ has no $pq$-element if and only if $G$ has only elements of prime power orders, thus Higman's theorem \cite{Higman} yields $G = P \rtimes Q$ ($P, Q$ possibly interchangeable) and $Q$ acts fixed-point-freely on $P$. Applying Theorem 8.1.12 in \cite{Kurz}, we know $G$ is a Frobenius group.  Conversely, a $pq$-Frobenius group obviously has no element of order $pq$.
\end{proof}

By Problem 6.16 of \cite{Is}, it follows that $\abs{Q} < \frac{1}{2} \abs{P}$ for the odd Frobenius group $G = P\rtimes Q$.
\begin{corollary}\label{pqt2}
Let $G$ be a $\{p, q\}$-separable group, and $p, q$ be distinct odd prime divisors of $\abs{G}$. Then $G$ contains no element of order $pq$ if and only if Hall $\{p, q\}$-subgroups of $G$ are Frobenius groups.
\end{corollary}
\begin{proof}
Omitted.
\end{proof}

  As shown above, $pq$-group need not contain element of order $pq$ but this does not really affect the existence of elements of order $pq$ in the difference set $N_{A_{2n}}(Q) - N_{A_{2n-1}}(Q)$, the following result indeed shows both $N_{A_{2n}}(Q)$ and $N_{A_{2n-1}}(Q)$ has the same Frobenius subgroups, which is Part 6 of \textsc{Theorem A}.

\begin{theorem}\label{norver2}
The even number $2n \geq 8$ is a Goldbach's number if and only if there exists an odd prime $n < q \leq 2n-3$ such that $\abs{\pi(N_{A_{2n}}(Q)) - \pi(N_{A_{2n-1}}(Q))} \geq 1$ for some $Q \in Syl_q(A_{2n})$.
\end{theorem}
\begin{proof}
For the prime  $n < q \leq 2n-3$  and $Q \in Syl_p(A_{2n})$, the application of  Lemma \ref{cean} yields that $$\pi(N_{A_{2n}}(Q)) = \{q \} \cup \pi(C_{q-1}) \cup \pi(A_{(q+1, 2n)})$$and  $$ \pi(N_{A_{2n-1}}(Q)) = \{q \} \cup \pi(C_{q-1}) \cup \pi(A_{(q+1, 2n-1)}),$$ thus we have that $$\pi(N_{A_{2n}}(Q)) - \pi(N_{A_{2n-1}}(Q)) = \pi(A_{(q+1, 2n)}) - \pi(A_{(q+1, 2n-1)}).$$
Using Lemma \ref{cean} again, it follows that $$\pi(C_{A_{2n}}(Q)) = \{q \} \cup \pi(A_{(q+1, 2n)}) \mbox{ and }  \pi(C_{A_{2n-1}}(Q)) = \{q \} \cup \pi(A_{(q+1, 2n-1)}),$$ hence we get that $$\pi(C_{A_{2n}}(Q)) - \pi(C_{A_{2n-1}}(Q)) = \pi(A_{(q+1, 2n)}) - \pi(A_{(q+1, 2n-1)}).$$Therefore, we conclude that $$\abs{\pi(N_{A_{2n}}(Q)) - \pi(N_{A_{2n-1}}(Q))} = \abs{\pi(C_{A_{2n}}(Q)) - \pi(C_{A_{2n-1}}(Q))}.$$
Corollary \ref{cenver2} implies the desired result.
\end{proof}

The next result may be compared with Corollary \ref{cenver3} replacing $C_{A_{2n-1}}(Q)$ and $C_{A_{2n}}(Q)$ by $N_{A_{2n-1}}(Q)$ and $N_{A_{2n}}(Q)$, respectively.

\begin{theorem}\label{norver3}
The even number $2n \geq 8$ is a Goldbach's number if and only if there exists an odd prime $n < q \leq 2n-3$ such that $N_{A_{2n}}(Q)$ contains elements of order $pq$ but $N_{A_{2n-1}}(Q)$ does not contains elements of order $pq$ for some $q \neq p \in \pi(A_{2n-1})$.
\end{theorem}
\begin{proof}
If $2n \geq 8$ is a Goldbach's number, then we may write $2n = q + p$ for odd primes $q > p$. Set $x = (1, 2, \cdots, q)$ and $Q = \langle x \rangle$. Corollary \ref{cenver2} yields
 $p$ divides $\abs{C_{A_{2n}}(Q)}$ and so there exists element $g$ of order $pq$ with $ g \in C_{A_{2n}}(Q) \leq N_{A_{2n}}(Q)$  but $p$ does not divides $\abs{C_{A_{2n-1}}(Q) }$. (Thus $C_{A_{2n-1}}(Q)$ does not contain any element of order $pq$.) If $N_{A_{2n-1}}(Q)$ has elements of order $pq$, then Lemma \ref{cean} implies $pq$ divides $\abs{Q \rtimes C_{q-1}}$, and
 $Q \rtimes C_{q-1}$ contains elements of order $pq$. However, this is impossible since Lemma \ref{cean} also yields $Q \rtimes C_{q-1}$ is a Frobenius group and Corollary \ref{pqt2} shows
 $Q \rtimes C_{q-1}$ has no element of order $pq$. The proof is completed.
\end{proof}
We mention that $A_n$ ($n \geq 4$) can  be characterized by the full set of orders of  normalizers of its Sylow's subgroups as stated in Theorem 1 of \cite{B57}.
\section{Group algebra}
In fact, the strongly binary Goldbach's conjecture is also expressed in the language of group algebra. Let $G$ be a finite group, and $p, q \in \pi(G)$ be distinct odd primes. Set $$s(p,q,g) = \underset{\underset{o(g) = pq}{x \in G}}{\Sigma}g^x.$$ It is easy to see that $s(p,q,g) \in Z(\Bbb{C}[G])$,where $Z(\Bbb{C}[G])$ denotes the center of group algebra $ \Bbb{C}[G]$ of the group $G$ over complex field $\Bbb{C}$ and it is a subalgebra of the group algebra $\Bbb{C}[G]$. By Theorem 2.4 of \cite{Is}, we see that all  conjugacy class sums of $G$ form a basis of $Z(\Bbb{C}[G])$.  The space linearly spanned by all possible $s(p,q,g)$'s is written as $\mathscr{U}\hspace{-0.5mm}(G) = \mathscr{L}(s(p,q,g) \mid g \in G, \mbox{distinct odd primes} \, p, q \in \pi(G))$ and we call $\mathscr{U}\hspace{-0.5mm}(G)$ as  \emph{the biprimary space} of $G$, then $\mathscr{U}\hspace{-0.5mm}(G)$ is a subspace of $ Z(\Bbb{C}[G])$. It is evident that $dim(\mathscr{U}\hspace{-0.5mm}(A_{2n-1})) \leq dim(\mathscr{U}\hspace{-0.5mm}(A_{2n}))$. Fix $$\tilde{s}(p,q,g) = \underset{\underset{o(g) = pq}{g \in A_{2n-1}, x \in A_{2n}}}{\Sigma}g^x,$$ and $$\tilde{\mathscr{U}}\hspace{-0.5mm}(A_{2n-1}) = \mathscr{L}(\tilde{s}(p,q,g) \mid g \in A_{2n-1}, \mbox{distinct odd primes} \, p, q \in \pi(A_{2n-1}))$$

It is clear that $\tilde{\mathscr{U}}\hspace{-0.5mm}(A_{2n-1})  \leq \mathscr{U}\hspace{-0.5mm}(A_{2n})$ and  $dim \mathscr{U}\hspace{-0.5mm}(A_{2n-1})  = dim \tilde{\mathscr{U}}\hspace{-0.5mm}(A_{2n-1}).$

The above observations imply the next result, which covers Part 7 of \textsc{Theorem} A.
\begin{theorem}\label{spdim}
There exist different odd primes $p, q $ with $2n = p + q$  $\iff  dim(\mathscr{U}\hspace{-0.5mm}(A_{2n-1})) < dim(\mathscr{U}\hspace{-0.5mm}(A_{2n}))$
$\iff \tilde{\mathscr{U}}\hspace{-0.5mm}(A_{2n-1}) < \mathscr{U}\hspace{-0.5mm}(A_{2n}) $
\end{theorem}
\begin{proof}
Omitted.
\end{proof}

In fact, the basis vectors $s(p,q,g)'s$ are computable. Set $s_1, s_2, \cdots, s_r$ to be a basis of the biprimary space $\mathscr{U}\hspace{-0.5mm}(A_n)$ ($n \geq 8$) and each $s_i$ stands for some $s(p,q,g)$ in a suitable order. Let $\{ e_1, e_2, \cdots, e_t \}$ be the full set of central primitive idempotent elements of group algebra $\Bbb{C}[A_n]$, let $\{ K_1, K_2, \cdots, K_t \}$ be the full set of class sums of $A_n$,
then we have $$(s_1, s_2, \cdots, s_r) = (K_1, K_2, \cdots, K_t)A_{tr},$$
where the $(i,j)$-entries $a_{ij}$ of $A_{tr}$ is either $1$ or else $0$.

The application of Theorem 2.12 of \cite{Is}, we conclude $$(e_1, e_2, \cdots, e_t) = (K_1, K_2, \cdots, K_t) \frac{1}{\abs{A_n}} \overline{X}^TC,$$
where $\overline{X}$ is the complex conjugate matrix of $X$ and $X$ is the character table of $A_n$, which is viewed as a matrix; the superscript $T$ denotes transpose; and $C$ is a diagonal matrix whose diagonal entries are all degrees $\chi_i(1)$ of irreducible characters $\chi_i$ of $A_n$. By the proof of Theorem 2.18 of \cite{Is}, we know $\abs{A_n}I = XD\overline{X}^T$, where $I$ is the identity matrix, $D$ is a diagonal matrix whose diagonal entries are the sizes $\abs{\mathscr{K}_i}$ of conjugacy classes $\mathscr{K}_i$ of $A_n$,  thus $(\overline{X}^T)^{-1} = \frac{1}{\abs{A_n}}XD$. We may deduce $$(K_1, K_2, \cdots, K_t) = (e_1, e_2, \cdots, e_t)C^{-1}XD$$ and so

$$(s_1, s_2, \cdots, s_r) = (e_1, e_2, \cdots, e_t)C^{-1}XDA_{tr}$$
 We may further derive

$$(s_1, s_2, \cdots, s_r) = (e_1, e_2, \cdots, e_t) M_{tr},$$

$$M_{tr} = \left ( \begin{array}{ccccc}
\frac{\chi_1(s_1)}{\chi_1(1)} & \frac{\chi_1(s_2)}{\chi_1(1)}& \cdots&\cdots& \frac{\chi_1(s_r)}{\chi_1(1)}\\
\frac{\chi_2(s_1)}{\chi_2(1)} &  \frac{\chi_2(s_2)}{\chi_2(1)}&\cdots&\cdots&\frac{\chi_2(s_r)}{\chi_2(1)}  \\
\vdots&\vdots&\ddots&&\vdots\\
\vdots&\vdots&&\ddots&\vdots\\
\frac{\chi_t(s_1)}{\chi_t(1)}& \frac{\chi_t(s_2)}{\chi_t(1)}&\cdots&\cdots& \frac{\chi_t(s_r)}{\chi_t(1)}\\
\end{array}
\right).$$

Theorem 3.7 of \cite{Is} implies the entries $\frac{\chi_i(s_j)}{\chi_i(1)}$ are algebraic integers. The $\chi_i 's$ are all irreducible characters of $A_n$. Since the two sets of  vectors are linearly independent respectively, it follows that the rank $R(M_{rt})$ of $M_{rt}$ is equal to $r$.

Using character theory, we may also extract some more specific information regarding elements of order $pq$ in $G$. For examples, if there exists an irreducible character of $G$ which is neither $p$-rational nor $q$-rational, then $G$ has $pq$-elements, which is a variation of Lemma 14.2 of \cite{Is}.

\textsc{Proof of Theorem} A.  Follows from  Theorem \ref{goldb3}, Corollary \ref{goldb4}, Theorem \ref{reduc}, Corollary \ref{cenver2}, Theorem \ref{norver2} and Theorem \ref{spdim}.

\textsc{Proof of Theorem} B.  Follows from  Theorems \ref{go-sp1} and \ref{go-sp2}.

\bigskip

\noindent \textbf{Acknowledgments} \quad This work was supported by National Natural Science Foundation of China (Grant No.11471054).
The first author wishes to thank  Prof. J. X. Bi for many helpful conversations on element order sets of finite simple groups and almost simple groups, especially on $A_n$ and $S_n$ for $n \geq 5$.

\bibliographystyle{amsalpha}

\bigskip
\begin{appendix}
\textbf{\large Appendix to GAP command codes}

The following GAP function is applied to computing the element order set of $A_n$.
\begin{verbatim}
jsqa:=function(n)
 local Al, ccreps, L, S;
 Al:=ConjugacyClasses(AlternatingGroup(n));
 ccreps:=List(Al, Representative);
 L:=List(ccreps, Order);
 S:=Set(L);
 return S;
 end;
\end{verbatim}
The following GAP command codes are used to compute $\abs{\mathcal{E}(A_{i}) - \mathcal{E}(A_{i-1})}$ when $5 \leq i \leq 60$.
\begin{verbatim}
> for i in [5..60] do
>  g := Difference(jsqa(i),jsqa(i-1));
>  for x in g do
>   ord := x;
>   n := FactorsInt(ord);
>   s := Size(n);
>   if s = 2 then
>   if n[1] <> n[2] then
>  Print(ord,":");
>  fi;fi;od;
>  Print( "::","\n");
>  od;
\end{verbatim}
\end{appendix}

\end{document}